\documentclass[12pt,twoside]{amsart}

\usepackage{amsmath,amsthm,amscd,amssymb,mathrsfs,graphicx,amsfonts,mathrsfs}
\usepackage{amsfonts}
\usepackage{amssymb,enumerate}
\usepackage{amsthm}
\usepackage[all]{xy}
\usepackage{hyperref}
\allowdisplaybreaks[4] \footskip=12pt
\renewcommand{\uppercasenonmath}[1]{}

\numberwithin{equation}{section} \theoremstyle{plain}
\newtheorem*{thm*}{Main Theorem}
\newtheorem{thm}{Theorem}[section]
\newtheorem{cor}[thm]{Corollary}
\newtheorem*{cor*}{Corollary}
\newtheorem{lem}[thm]{Lemma}
\newtheorem*{lem*}{Lemma}
\newtheorem{fact}[thm]{Fact}
\newtheorem*{fact*}{Fact}

\newtheorem*{nota*}{Notation}
\newtheorem{prop}[thm]{Proposition}
\newtheorem*{prop*}{Proposition}

\newtheorem*{rem*}{Remark}

\newtheorem*{observation*}{Observation}

\newtheorem*{exa*}{Example}

\newtheorem*{df*}{Definition}

\newtheorem*{conj*}{Conjecture}

\renewcommand{\geq}{\geqslant}
\renewcommand{\leq}{\leqslant}

\begin{document}
\begin{center}
{\large  \bf Coartinianess of local homology modules for ideals of small dimension}

\vspace{0.3cm} Jingwen Shen, Pinger Zhang, Xiaoyan Yang \\
Department of Mathematics, Northwest Normal University, Lanzhou 730070,
China\\
E-mails: shenjw0609@163.com, 275661482@qq.com, yangxy@nwnu.edu.cn
\end{center}
\bigskip
\centerline { \bf  Abstract} Let $\mathfrak{a}$ be an ideal of a commutative noetherian ring $R$ and $M$ an $R$-module with Cosupport in $\mathrm{V}(\mathfrak{a})$. We show that $M$ is $\mathfrak{a}$-coartinian if and only if $\mathrm{Ext}_{R}^{i}(R/\mathfrak{a},M)$ is artinian for all $0\leq i\leq \mathrm{cd}(\mathfrak{a},M)$, which provides a computable finitely many steps to examine $\mathfrak{a}$-coartinianness. We also consider the duality of Hartshorne's questions: for which rings $R$ and ideals $\mathfrak{a}$ are the modules $\mathrm{H}^{\mathfrak{a}}_{i}(M)$ $\mathfrak{a}$-coartinian for all $i\geq 0$; whether the category $\mathcal{C}(R,\mathfrak{a})_{coa}$ of $\mathfrak{a}$-coartinian modules is an Abelian subcategory of the category of all $R$-modules, and establish affirmative answers to these questions in the case $\mathrm{cd}(\mathfrak{a},R)\leq 1$ and $\mathrm{dim}R/\mathfrak{a}\leq 1$.
\leftskip10truemm \rightskip10truemm \noindent \\
\vbox to 0.3cm{}

\leftskip0truemm \rightskip0truemm
\bigskip

\section{\bf Introduction and Preliminaries}

The theory of local cohomology has been developed rapidly for the last $\mathrm{60}$ years after it was introduced by Grothendieck. Let $R$ denote a commutative noetherian ring with identity and $\mathfrak{a}$ an ideal of $R$. For an $R$-module $M$, the $i$th local cohomology module of $M$ with respect to $\mathfrak{a}$ is defined as
\begin{center}$\begin{aligned}
\mathrm{H}^{i}_{\mathfrak{a}}(M)=\underrightarrow{\text{lim}}\mathrm{Ext}^{i}_{R}(R/\mathfrak{a}^{t},M).
\end{aligned}$\end{center}
Local homology as its duality was initiated by Matlis \cite{M4} in 1974. Denote $\Lambda^{\mathfrak{a}}(M)=\underleftarrow{\text{lim}}M/\mathfrak{a}^{t}M$ is the $\mathfrak{a}$-adic completion of $M$. Even if $R$ is noetherian, the $\mathfrak{a}$-adic functor $\Lambda^{\mathfrak{a}}(-)$ is neither left nor right exact on the category of all $R$-modules. Hence the computation of the left derived functors $L_{\bullet}^{\mathfrak{a}}(-)$ of the $\mathfrak{a}$-adic functor $\Lambda^{\mathfrak{a}}(-)$ is very difficult and so local homology is not known so much. Moreover, we recall the $i$th local homology module of $M$ is
\begin{center}$\begin{aligned}
\mathrm{H}_{i}^{\mathfrak{a}}(M)=\underleftarrow{\text{lim}}\mathrm{Tor}_{i}^{R}(R/\mathfrak{a}^{t},M).
\end{aligned}$\end{center}
Cuong and Nam \cite{CN1} showed that $\Lambda^{\mathfrak{a}}(-)$ has right exactness on the category of semi-discrete linearly compact $R$-modules, which says that $\Lambda^{\mathfrak{a}}(M)=\mathrm{H}_{0}^{\mathfrak{a}}(M)\cong L_{0}^{\mathfrak{a}}(M)$ and $\mathrm{H}_{i}^{\mathfrak{a}}(M)\cong L_{i}^{\mathfrak{a}}(M)$ for $i> 0$ when $M$ is a semi-discrete linearly compact $R$-module.

Let $(R,\mathfrak{m},k)$ be a complete local ring, $E$ an injective hull of $k$ over $R$. We recall the following various results due to Matlis and Grothendieck.

\begin{fact}\label{fact:1.1}\rm
For an $R$-module $N$, the following conditions are equivalent.

$\mathrm{(1)}$ $N$ is artinian.

$\mathrm{(2)}$ $N$ is isomorphic to a submodule of a finite direct sum of copies of $E$.

$\mathrm{(3)}$ There is a finitely generated $R$-module $M$ such that $N\cong \mathrm{Hom}_{R}(M,E)$.

$\mathrm{(4)}$ $\mathrm{Supp}_{R}N=\{\mathfrak{m}\}$ and $\mathrm{Hom}_{R}(k,N)$ is finitely generated.
\end{fact}

Hartshorne \cite{H} called $N$ $\mathfrak{m}$-cofinite if it satisfies the equivalent conditions of the above fact, and the $\mathfrak{m}$-cofinite modules form an abelian subcategory of the category of all $R$-modules. For an ideal $\mathfrak{a}$ of $R$, Hartshorne defined an $R$-module $L$ to be $\mathfrak{a}$-cofinite, if $\mathrm{Supp}_{R}(L)\subseteq \mathrm{V}(\mathfrak{a})$ and $\mathrm{Ext}^{i}_{R}(R/\mathfrak{a},L)$ is finitely generated for all $i$, which is to generalize condition $\mathrm{(4)}$ of the fact. Cofinite module is an important tool in the studying of Huneke's problems on local cohomology. Hartshorne also posed two questions on cofinite modules:

{\bf Question 1.} \textit{Are the local cohomology modules $\mathrm{H}_{\mathfrak{a}}^{i}(M)$ $\mathfrak{a}$-cofinite for every finitely generated $R$-module $M$ and every $i\geq 0$}?

{\bf Question 2.} \textit{Whether the category $\mathcal{C}(R,\mathfrak{a})_{cof}$ of $\mathfrak{a}$-cofinite modules is an Abelian subcategory of the category of all $R$-modules}?

Question 1 and 2 have been high-profile among researchers. For example, see \cite{B, BN, C, DM, DFT, K, M2, MM, Y3}.

Combing Fact \ref{fact:1.1} with Matlis duality, it is clear when $(R,\mathfrak{m})$ is a complete local ring, an $R$-module $N$ is finitely generated is equivalent to $\mathrm{Cosupp}_{R}N\subseteq \{\mathfrak{m}\}$ and $N/\mathfrak{m}N$ is artinian. Hence Nam \cite{N} defined $\mathfrak{a}$-coartinian module which is in some sense dual to the concept of cofinite module. Recall that an $R$-module $M$ is $\mathfrak{a}$-coartinian if $\mathrm{Cosupp}_{R}M\subseteq \mathrm{V}(\mathfrak{a})$, and $\mathrm{Tor}_{i}^{R}(R/\mathfrak{a},M)$ is artinian for all $i\geq 0$.

The first aim of this paper is to give an equivalent characterization of coartinian modules. More precisely, we show that if $M$ is an $R$-module with $\mathrm{Cosupp}_{R}M\subseteq \mathrm{V}(\mathfrak{a})$, then $M$ is $\mathfrak{a}$-coartinian if and only if $\mathrm{Ext}_{R}^{i}(R/\mathfrak{a},M)$ is artinian for all $0\leq i\leq \mathrm{cd}(\mathfrak{a},M)$ in section 2.

Nam \cite{N1} posed a question on local homology: whether the number of coassociated primes of local homology modules $\mathrm{H}^{\mathfrak{a}}_{i}(M)$ always finite for every $i\geq 0$. One can easily observe that if the local homology $R$-module $\mathrm{H}_{i}^{\mathfrak{a}}(M)$ is $\mathfrak{a}$-coartinian, then the set $\mathrm{Coass}_{R}\mathrm{H}_{i}^{\mathfrak{a}}(M)$ and the Bass number $\mu_{\hat{R}}^{i}(\mathfrak{P},\mathrm{D}(\mathrm{H}_{i}^{\mathfrak{a}}(M)))$ are finite for any $\mathfrak{P}\in \mathrm{Spec}\hat{R}$, when $(R,\mathfrak{m})$ is local with $\mathrm{D}(-)=\mathrm{Hom}_{R}(-,E(R/\mathfrak{m}))$ and $\hat{R}$ is the $\mathfrak{m}$-adic completion of $R$.

In section 3, we consider the following two questions on coartinian modules which are dual to Hartshorne's questions:

{\bf Question $\mathbf{1}'$.} \textit{For which rings $R$ and ideals $\mathfrak{a}$ are the modules $\mathrm{H}^{\mathfrak{a}}_{i}(M)$ $\mathfrak{a}$-coartinian for all $i\geq 0$}?

{\bf Question $\mathbf{2}'$.} \textit{Whether the category of $\mathfrak{a}$-coartinian modules is an Abelian subcategory of the category of all $R$-modules}?

For the first question, we show that when $(R,\mathfrak{m})$ is a local ring with the $\mathfrak{m}$-adic topology and $M$ a semi-discrete linearly compact $R$-module with $\mathrm{mag}_{R}(M)\leq 1$, then $\mathrm{H}_{i}^{\mathfrak{a}}(M)$ is $\mathfrak{a}$-coartinian for all $i$. Subsequently, we answer Question $2'$ completely in the case $\mathrm{cd}(\mathfrak{a},R)\leq 1$ and $\mathrm{dim}R/\mathfrak{a}\leq 1$.

Next we recall some notions which will need later.

We write $\mathrm{Spec}R$ for the set of prime ideals of $R$, $\mathrm{Max}R$ for the set of maximal ideals of $R$. For an ideal $\mathfrak{a}$ of $R$, set

\begin{center}$\begin{aligned}
\mathrm{V}(\mathfrak{a})=\{\mathfrak{p}\in \mathrm{Spec}R\hspace{0.03cm}|\hspace{0.03cm}\mathfrak{a}\subseteq \mathfrak{p}\}.
\end{aligned}$\end{center}

Fix $\mathfrak{p}\in \mathrm{Spec}R$, let $M_{\mathfrak{p}}$ denote the localization of $M$ at $\mathfrak{p}$.

{\bf Associated prime and Coassociated prime.} Let $M$ be an $R$-module. The set of associated prime ideals of $M$ is denoted by $\mathrm{Ass}_{R}M$ and it is the set of primes $\mathfrak{p}$ such that there exists a cyclic submodule $N$ of $M$ with $\mathfrak{p}=\mathrm{Ann}_{R}N$.

Yassemi \cite{Y1} introduced the cocyclic modules. An $R$-module $L$ is cocyclic if $L$ is a submodule of $E(R/\mathfrak{m})$ for some $\mathfrak{m}\in \mathrm{Max}R$. A prime ideal $\mathfrak{p}$ of $R$ is called a coassociated prime of $R$-module $M$ if there exists a cocyclic homomorphic image $L$ of $M$ such that $\mathfrak{p}=\mathrm{Ann}_{R}L$, the annihilator of $L$. The set of coassociated prime ideals of $M$ is denoted by $\mathrm{Coass}_{R}M$.

{\bf Support and Cosupport.} The ``large" support of an $R$-module $M$ which is denoted by $\mathrm{Supp}_{R}M$ is defined as the set of prime ideals of $\mathfrak{p}$ such that there is a cyclic submodule $N$ of $M$ with $\mathrm{Ann}_{R}N\subseteq \mathfrak{p}$.

The cosupport of $M$ is defined as the set of prime ideals $\mathfrak{p}$ such that there is a cocyclic homomorphic image $L$ of $K$ with $\mathfrak{p}\supseteq \mathrm{Ann}_{R}L$, and denoted this set by $\mathrm{Cosupp}_{R}M$.

{\bf Dimension and Magnitude.} The (Krull) dimension of an $R$-module $0\neq M$ is
\begin{center}
$\mathrm{dim}_{R}M=\mathrm{sup}\{\mathrm{dim}R/\mathfrak{p}\hspace{0.03cm}|\hspace{0.03cm}\mathfrak{p}\in \mathrm{Supp}_{R}M\}$.
\end{center}

If $M=0$, we write $\mathrm{dim}_{R}M=-\infty$.

Yassemi \cite{Y2} introduced a dual concept of dimension for modules and called it magnitude of modules and denoted $\mathrm{mag}_{R}M$.
\begin{center}$\begin{aligned}
\mathrm{mag}_{R}M=\mathrm{sup}\{\mathrm{dim}R/\mathfrak{p}\hspace{0.03cm}|\hspace{0.03cm}\mathfrak{p}\in \mathrm{Cosupp}_{R}M\}.
\end{aligned}$\end{center}

If $M=0$, we write $\mathrm{mag}_{R}M=-\infty$.

{\bf Cohomological dimension and homological dimension.} Let $\mathfrak{a}$ be an ideal of $R$, $M$ an $R$-module. Recall that the cohomological dimension of $M$ with respect to $\mathfrak{a}$ is
\begin{center}
$\mathrm{cd}(\mathfrak{a},M)=
\mathrm{sup}\{i\geq0\hspace{0.03cm}|\hspace{0.03cm}
\mathrm{H}_{\mathfrak{a}}^{i}(M)\neq0\}$.
\end{center}

The homological dimension of $\mathfrak{a}$ with respect to $M$ is
\begin{center}
$\mathrm{hd}(\mathfrak{a},M)=
\mathrm{sup}\{i\geq0\hspace{0.03cm}|\hspace{0.03cm}
\mathrm{H}^{\mathfrak{a}}_{i}(M)\neq0\}$.
\end{center}

Throughout this paper, $R$ is a commutative noetherian ring with non-zero identity and $\mathfrak{a}$ an ideal of $R$.

\bigskip
\section{\bf Coartinianness of modules }

This section will provide a computable finitely many steps to examine $\mathfrak{a}$-coartinianness for an $R$-module $M$ with $\mathrm{Cosupp}_{R}M\subseteq \mathrm{V}(\mathfrak{a})$.

\begin{lem}\label{lem:2.1}
Let $\mathfrak{a}$ be an ideal of $R$, $M$ an $R$-module and $s> 0$. If $\mathrm{Ext}_{R}^{i}(R/\mathfrak{a},\mathrm{H}_{\mathfrak{a}}^{j}(M))$ is artinian for $i\geq 0$ and $0\leq j\leq s-1$ and $\mathrm{Ext}_{R}^{s}(R/\mathfrak{a},M)$ is artinian, then $\mathrm{H}_{\mathfrak{a}}^{s}(M)$ is artinian.
\end{lem}

\begin{proof}
There is a spectral sequence
$$E^{p,q}_{2}=\mathrm{Ext}_{R}^{p}(R/\mathfrak{a},\mathrm{H}_{\mathfrak{a}}^{q}(M))\Rightarrow \mathrm{Ext}_{R}^{p+q}(R/\mathfrak{a},M).$$
For $s\geq 0$, there is a finite filtration
$$0=U^{-1}\subseteq U^{0}\subseteq \cdots \subseteq U^{s-1}\subseteq U^{s}=\mathrm{Ext}_{R}^{s}(R/\mathfrak{a},M),$$
such that $U^{p}/U^{p-1}\cong E^{p,s-p}_{\infty}$ for every $0\leq p\leq s$. As $\mathrm{Ext}_{R}^{s}(R/\mathfrak{a},M)$ is artinian, we have $E^{0,s}_{\infty}\cong U^{0}/U^{-1}$ is artinian. Let $r\geq 2$, consider the differentials
$$\xymatrix{0=E^{-r,s+r-1~}_{r}\ar[r]^-{d^{-r,s+r-1}_{r}} & E^{0,s}_{r}\ar[r]^-{d^{0,s}_{r}} & E^{r,s-r+1}_{r},}$$
where the vanishing comes from the facts that $E^{-r,s+r-1}_{2}=0$ and $E^{-r,s+r-1}_{r}$ is a subquotient of $E^{-r,s+r-1}_{2}$ for all $r\geq2$. On the other hand, we obtain $E_{0,s}^{r+1}=\mathrm{Ker}d^{0,s}_{r}/\mathrm{Im}d^{-r,s+r-1}_{r}=\mathrm{Ker}d^{0,s}_{r}$ and a short exact sequence
$$0\rightarrow E^{0,s}_{r+1}\rightarrow E^{0,s}_{r}\rightarrow \mathrm{Im}d^{0,s}_{r}\rightarrow 0.$$
As $s-r+1\leq s-1$, the hypothesis implies that $E^{r,s-r+1}_{r}$ is artinian, consequently $\mathrm{Im}d^{0,s}_{r}$ is artinian. Note that $E^{0,s}_{\infty}=E^{0,s}_{r+1}$ for every $r\geq r_{0}=s+2$. It follows that $E^{0,s}_{r_{0}+1}$ is artinian. Hence the above short exact sequence implies that $E^{0,s}_{r_{0}}$ is artinian. Using the above short exact sequence inductively, we conclude that $E^{0,s}_{2}=\mathrm{Hom}_{R}(R/\mathfrak{a},\mathrm{H}_{\mathfrak{a}}^{s}(M))$ is artinian. Hence, \cite[Theorem 1.3]{M1} implies that $\mathrm{H}_{\mathfrak{a}}^{s}(M)$ is artinian.
\end{proof}

The following Lemma is crucial to the proof of the main result in the section.

\begin{lem}\label{lem:2.2}
Let $M$ be an $R$-module. Then the following conditions are equivalent.

$\mathrm{(1)}$ $\mathrm{H}_{\mathfrak{a}}^{i}(M)$ is artinian for all $i\geq 0$.

$\mathrm{(2)}$ $\mathrm{Ext}_{R}^{i}(R/\mathfrak{a},M)$ is artinian for all $i\geq 0$.

$\mathrm{(3)}$ $\mathrm{Ext}_{R}^{i}(N,M)$ is artinian for every finitely generated $R$-module $N$ with $\mathrm{Supp}_{R}N\subseteq \mathrm{V}(\mathfrak{a})$ and every $i\geq 0$.

$\mathrm{(4)}$ $\mathrm{Ext}_{R}^{i}(R/\mathfrak{a},M)$ is artinian for all $0\leq i\leq \mathrm{cd}(\mathfrak{a},M)$.
\end{lem}

\begin{proof}
$\mathrm{(1)}\Rightarrow \mathrm{(2)}$ Holds by \cite[Proposition 3.3]{B}.

$\mathrm{(2)}\Rightarrow \mathrm{(3)}$ Let $N$ be a finitely generated $R$-module with $\mathrm{Supp}_{R}N\subseteq \mathrm{V}(\mathfrak{a})$. We argue by induction on $i$. Let $i=0$, we show $\mathrm{Hom}_{R}(N,M)$ is artinian. Using the improved consequence of Gruson's Theorem \cite[Lemma 2.2]{B}, there is a finite filtration
$$0=N_{0}\subseteq N_{1}\subseteq \cdots\subseteq N_{t-1}\subseteq N_{t},$$
such that $N_{j}/N_{j-1}$ is a quotient of $R/\mathfrak{a}$ for $1\leq j\leq t$. For $1\leq j\leq t$, the exact sequence
$$0\rightarrow K_{j}\rightarrow R/\mathfrak{a}\rightarrow N_{j}/N_{j-1}\rightarrow 0,$$
induces an exact sequence
$$0\rightarrow \mathrm{Hom}_{R}(N_{j}/N_{j-1},M)\rightarrow \mathrm{Hom}_{R}(R/\mathfrak{a},M).$$
It implies that $\mathrm{Hom}_{R}(N_{j}/N_{j-1},M)$ is artinian. Now the exact sequence
$$0\rightarrow N_{j-1}\rightarrow N_{j}\rightarrow N_{j}/N_{j-1}\rightarrow 0,$$
yields the exact sequence
$$0\rightarrow \mathrm{Hom}_{R}(N_{j}/N_{j-1},M)\rightarrow \mathrm{Hom}_{R}(N_{j},M)\rightarrow \mathrm{Hom}_{R}(N_{j-1},M).$$
Using the above exact sequences successively, we have $\mathrm{Hom}_{R}(N,M)$ is artinian. Now assume that the result holds for less than $i$, it suffices to show $\mathrm{Ext}_{R}^{i}(N,M)$ is artinian. For $1\leq j\leq t$, there is an exact sequence
$$\mathrm{Ext}_{R}^{i-1}(K_{j},M)\rightarrow \mathrm{Ext}_{R}^{i}(N_{j}/N_{j-1},M)\rightarrow \mathrm{Ext}_{R}^{i}(R/\mathfrak{a},M).$$
The induction hypothesis shows that $\mathrm{Ext}_{R}^{i}(K_{j},M)$ is artinian since $K_{j}$ is finitely generated and $\mathrm{Supp}_{R}K_{j}\subseteq \mathrm{V}(\mathfrak{a})$. From the above exact sequence we get $\mathrm{Ext}_{R}^{i}(N_{j}/N_{j-1},M)$ is artinian. Note that there is also an exact sequence
$$\mathrm{Ext}_{R}^{i}(N_{j}/N_{j-1},M)\rightarrow \mathrm{Ext}_{R}^{i}(N_{j},M)\rightarrow \mathrm{Ext}_{R}^{i}(N_{j-1},M).$$
A successive use of the above exact sequences implies that $\mathrm{Ext}_{R}^{i}(N,M)$ is artinian.

$\mathrm{(3)}\Rightarrow \mathrm{(2)}$ and $\mathrm{(2)}\Rightarrow \mathrm{(4)}$ Obvious.

$\mathrm{(4)}\Rightarrow \mathrm{(1)}$ Since $\mathrm{H}_{\mathfrak{a}}^{i}(M)=0$ is artinian for $i> \mathrm{cd}(\mathfrak{a},M)$, we use induction on $s=\mathrm{cd}(\mathfrak{a},M)$ to show that if $\mathrm{Ext}_{R}^{i}(R/\mathfrak{a},M)$ is artinian for $0\leq i\leq s$, then $\mathrm{H}_{\mathfrak{a}}^{i}(M)$ is artinian for $0\leq i\leq s$. If $s=0$, then $\mathrm{Hom}_{R}(R/\mathfrak{a},M)$ is artinian, it implies that $\mathrm{Hom}_{R}(R/\mathfrak{a},\mathrm{H}_{\mathfrak{a}}^{0}(M))$ is artinian because $\mathrm{H}_{\mathfrak{a}}^{0}(M)\cong \Gamma_{\mathfrak{a}}(M)$ is a submodule of $M$. Therefore, it follows from \cite[Theorem 1.3]{M1} that $\mathrm{H}_{\mathfrak{a}}^{0}(M)$ is artinian since $\mathrm{H}_{\mathfrak{a}}^{0}(M)$ is $\mathfrak{a}$-torsion. Now assume that $s> 0$ and the result holds for $s-1$. Then $\mathrm{H}_{\mathfrak{a}}^{j}(M)$ is artinian for $0\leq j\leq s-1$. So $\mathrm{Ext}_{R}^{i}(R/\mathfrak{a},\mathrm{H}_{\mathfrak{a}}^{j}(M))$ is artinian for $i\geq 0$ and $0\leq j\leq s-1$, and $\mathrm{Ext}_{R}^{s}(R/\mathfrak{a},M)$ is artinian. Hence Lemma \ref{lem:2.1} implies that $\mathrm{H}_{\mathfrak{a}}^{s}(M)$ is artinian.
\end{proof}

\begin{thm}\label{thm:2.3}
Let $\mathfrak{a}$ be an ideal of $R$ and $M$ an $R$-module with $\mathrm{Cosupp}_{R}M\subseteq \mathrm{V}(\mathfrak{a})$. Then the following conditions are equivalent.

$\mathrm{(1)}$ $M$ is $\mathfrak{a}$-coartinian.

$\mathrm{(2)}$ $\mathrm{Ext}_{R}^{i}(R/\mathfrak{a},M)$ is artinian for all $0\leq i\leq \mathrm{cd}(\mathfrak{a},M)$.
\end{thm}

\begin{proof}
According to \cite[Proposition 3.3]{B}, the $R$-module $\mathrm{Ext}_{R}^{i}(R/\mathfrak{a},M)$ is artinian if and only if $\mathrm{Tor}^{R}_{i}(R/\mathfrak{a},M)$ is artinian for each $i\geq 0$. Therefore, the proof is clear by Lemma \ref{lem:2.2}.
\end{proof}

\begin{lem}\label{lem:2.4}
Let $f$: $M\rightarrow N$ be an $R$-homomorphism. If both $\mathrm{Tor}^{R}_{i}(R/\mathfrak{a},\mathrm{ker}f)$ and $\mathrm{Tor}^{R}_{i}(R/\mathfrak{a},\mathrm{coker}f)$ are artinian, then $\mathrm{kerTor}^{R}_{i}(R/\mathfrak{a},f)$ and $\mathrm{cokerTor}^{R}_{i}(R/\mathfrak{a},f)$ are also artinian, for all $i \geq 0$.
\end{lem}

\begin{proof}
~It is clear by \cite[Lemma 3.1]{M2}.
\end{proof}

The following proposition is very useful to decide whether a module is coartinian.

\begin{prop}\label{prop:2.5}
Suppose that $x\in \mathfrak{a}$ and $\mathrm{Cosupp}_{R}M \subseteq \mathrm{V}(\mathfrak{a})$. If $(0:_{M}x)$ and $M/xM$ are both $\mathfrak{a}$-coartinian, then $M$ is also $\mathfrak{a}$-coartinian.
\end{prop}

\begin{proof}
Consider $R$-homomorphism $x$:$M\rightarrow M$. Then $\mathrm{ker}x=(0:_{M}x)$, $\mathrm{coker}x = M/xM$. Since $\mathrm{ker}x$ and $\mathrm{coker}x$ are $\mathfrak{a}$-coartinian, we have $\mathrm{Tor}_{i}^{R}(R/\mathfrak{a},\mathrm{ker}x)$ and $\mathrm{Tor}_{i}^{R}(R/I,\mathrm{coker}x)$ are artinian for all $i$. Apply Lemma \ref{lem:2.4} to $x$, $\mathrm{kerTor}_{i}^{R}(R/\mathfrak{a},x)$ and $\mathrm{cokerTor}_{i}^{R}(R/\mathfrak{a},x)$ are artinian for all $i$. While $x \in \mathfrak{a}$, $\mathrm{Tor}_{i}^{R}(R/\mathfrak{a},x)=0$ for all $i$, hence $\mathrm{kerTor}_{i}^{R}(R/\mathfrak{a},x)= \mathrm{Tor}_{i}^{R}(R/\mathfrak{a},M)$ is artinian. Thus $M$ is also $\mathfrak{a}$-coartinian.
\end{proof}

\bigskip
\section{\bf  Two questions on $\mathfrak{a}$-coartinian modules}

In this section, we give an affirmative answer for these two questions in introduction when $\mathrm{cd}(\mathfrak{a},R)\leq 1$ or $\mathrm{dim}R/\mathfrak{a}\leq 1$ in the category of semi-discrete linearly compact $R$-modules.

We begin by recalling definition  of linearly compact modules that we will use. Let $M$ be a topological $R$-module. $M$ is said to be \textit{linearly topologized} if $M$ has a base of neighborhoods of the zero element $\mathcal{M}$ consisting of submodules. $M$ is called \textit{Hausdorff} if the intersection of all the neighborhoods of the zero element is $0$. A Hausdorff linearly topologized $R$-module $M$ is said to be \textit{linearly compact} if $\mathcal{F}$ is a family of closed cosets (i.e., cosets of closed submodules) in $M$ which has the finite intersection property, then the cosets in $\mathcal{F}$ have a non-empty intersection (see \cite{M}). A Hausdorff linearly topologized $R$-module $M$ is called semi-discrete if every submodule of $M$ is closed. The class of semi-discrete linearly compact $R$-modules contains all artinian $R$-modules.

\begin{lem}\label{lem:3.1}
Let $M$ be a linearly compact $R$-module, then $\mathrm{H}_{i}^{\mathfrak{a}}(M)=0$ for all $i> \mathrm{mag}_{R}(M)$.
\end{lem}

\begin{proof}
It follows from \cite[Theorem 2.8]{M3}.
\end{proof}

\begin{lem}\label{lem:3.2}
Let $M$ be a linearly compact $R$-module such that $\mathrm{Tor}_{i}^{R}(R/\mathfrak{a},M)$ is artinian for every $i$. If $s$ is an integer such that $\mathrm{H}_{i}^{\mathfrak{a}}(M)$ is $\mathfrak{a}$-coartinian for all $i\neq s$, then $\mathrm{H}_{s}^{\mathfrak{a}}(M)$ is $\mathfrak{a}$-coartinian.
\end{lem}

\begin{proof}
Since $M$ is linearly compact, then $\mathrm{Cosupp}_{R}\mathrm{H}_{s}^{\mathfrak{a}}(M)\subseteq \mathrm{V}(\mathfrak{a})$. Hence we only need to show that $\mathrm{Tor}_{j}^{R}(R/\mathfrak{a},\mathrm{H}_{s}^{\mathfrak{a}}(M))$ is artinian for all $j\geq 0$. Note that there is a spectral sequence
$$E^{2}_{p,q}=\mathrm{Tor}_{p}^{R}(R/\mathfrak{a},\mathrm{H}_{q}^{\mathfrak{a}}(M))\Rightarrow \mathrm{Tor}_{p+q}^{R}(R/\mathfrak{a},M).$$
We argue by induction on $j$. Let $j=0$, we show that $R/\mathfrak{a}\underset{R}\otimes \mathrm{H}_{s}^{\mathfrak{a}}(M)$ is artinian. There exists a finite filtration
$$0=U^{-1}\subseteq U^{0}\subseteq \cdots\subseteq U^{s}=\mathrm{Tor}_{s}^{R}(R/\mathfrak{a},M),$$
such that $U^{p}/U^{p-1}\cong E_{p,s-p}^{\infty}$ for every $0\leq p\leq s$. As $\mathrm{Tor}_{s}^{R}(R/\mathfrak{a},M)$ is artinian, we conclude that $E_{0,s}^{\infty}\cong U^{0}/U^{-1}$ is artinian.
Assume that $r\geq 2$, consider the differentials
$$\xymatrix{E_{r,s-r+1}^{r}\ar[r]^{d_{r,s-r+1}^{r}} & E_{0,s}^{r}\ar[r]^-{d_{0,s}^{r}} & E_{-r,s+r-1}^{r}=0.}$$
Since $\mathrm{H}_{i}^{\mathfrak{a}}(M)$ is $\mathfrak{a}$-coartinian for $i\neq s$, $E_{r,s-r+1}^{2}$ is artinian. On the other hand, $E_{r,s-r+1}^{r}$ is a subquotient of $E_{r,s-r+1}^{2}$, it follows that $E_{r,s-r+1}^{r}$ is artinian and consequently $\mathrm{Im}d_{r,s-r+1}^{r}$ is also artinian. We obtain a short exact sequence
$$0\rightarrow \mathrm{Im}d_{r,s-r+1}^{r}\rightarrow E_{0,s}^{r}\rightarrow E_{0,s}^{r+1}\rightarrow 0.$$
There is an integer $r_{0}\geq 2$, such that $E_{r,s-r+1}^{r}=0$ for all $r\geq r_{0}$, hence $E_{0,s}^{r+1}\cong E_{0,s}^{\infty}$ for every $r\geq r_{0}$. It follows that $E_{0,s}^{r_{0}+1}$ is artinian. Now the short exact sequence implies that $E_{0,s}^{r_{0}}$ is artinian. Using the short exact sequence inductively, we conclude that $E_{0,s}^{2}=R/\mathfrak{a}\underset{R}\otimes \mathrm{H}_{s}^{\mathfrak{a}}(M)$ is artinian. Now suppose that the result holds for less than $j$, we show that $E^{2}_{j,s}=\mathrm{Tor}_{j}^{R}(R/\mathfrak{a},\mathrm{H}_{s}^{\mathfrak{a}}(M))$ is artinian. Let $r\geq 2$, consider the following differentials
$$\xymatrix{E_{j+r,s-r+1}^{r}\ar[r]^{d_{j+r,s-r+1}^{r}} & E_{j,s}^{r}\ar[r]^-{d_{j,s}^{r}} & E_{j-r,s+r-1}^{r}.}$$
Since $\mathrm{H}_{i}^{\mathfrak{a}}(M)$ is $\mathfrak{a}$-coartinian for $i\neq s$, $E_{j+r,s-r+1}^{2}$ and $E_{j-r,s+r-1}^{2}$ are artinian, while $E_{j+r,s-r+1}^{r}$ is a subquotient of $E_{j+r,s-r+1}^{2}$ and $E_{j-r,s+r-1}^{r}$ is a sunquotient of $E_{j-r,s+r-1}^{2}$, hence $E_{j+r,s-r+1}^{r}$ and $E_{j-r,s+r-1}^{r}$ are artinain, and thus $\mathrm{Im}d_{j+r,s-r+1}$ and $\mathrm{Im}d_{j,s}^{r}$ are artinian. There is $r_{1}\geq 2$, such that $E_{j+r,s-r+1}^{r}=0$, hence we have an exact sequence
$$0\rightarrow E_{j,s}^{r_{1}}\rightarrow  E_{j-r_{1},s+r_{1}-1}^{r_{1}},$$
and since $E_{j-r_{1},s+r_{1}-1}^{r_{1}}$ is artinian, it follows that $E_{j,s}^{r_{1}}$ is aritinian.
On the other hand, there are two short exact sequences
$$0\rightarrow \mathrm{Ker}d_{j,s}^{r}\rightarrow E_{j,s}^{r}\rightarrow \mathrm{Im}d_{j,s}^{r}\rightarrow 0,$$
$$0\rightarrow \mathrm{Im}d_{j+r,s-r+1}\rightarrow \mathrm{Ker}d_{j,s}^{r}\rightarrow E_{j,s}^{r+1}\rightarrow 0.$$
These two short exact sequences imply that $E_{j,s}^{r_{1}-1}$ is artinian. Using the short exact sequences inductively, we have $E_{j,s}^{2}=\mathrm{Tor}_{j}^{R}(R/\mathfrak{a},\mathrm{H}_{s}^{\mathfrak{a}}(M))$ is artinian.
\end{proof}

A sequence of elements $x_{1},\ldots, x_{n}$ in $R$ is said to be an $M$-coregular sequence if $(0:_{M}(x_{1},\ldots,x_{i-1}))\rightarrow (0:_{M}(x_{1},\ldots,x_{i-1}))$ is surjective and $(0:_{M}(x_{1},\ldots,x_{i}))\neq 0$ for $i=1,\ldots,n$. We denote by $\mathrm{width}_{\mathfrak{a}}(M)$ the supremum of the lengths of all maximal $M$-coregular sequences in the ideal $\mathfrak{a}$. Following \cite{CN1}, if $M$ is a semi-discrete linearly compact $R$-module, then
\begin{center}
$\mathrm{width}_{\mathfrak{a}}M=
\mathrm{inf}\{i\geq0\hspace{0.03cm}|\hspace{0.03cm}
\mathrm{H}^{\mathfrak{a}}_{i}(M)\neq0\}$.
\end{center}

\begin{cor}\label{cor:3.3}
If $M$ is a semi-discrete linearly compact $R$-module such that $\mathrm{H}^{\mathfrak{a}}_{i}(M)\neq0$ for all $i\neq s$. i.e. $\mathrm{width}_{\mathfrak{a}}M=\mathrm{hd}(\mathfrak{a},M)=s$, then $\mathrm{H}^{\mathfrak{a}}_{s}(M)$ is $\mathfrak{a}$-coartinian.
\end{cor}

\begin{proof}
It is clear from Lemma \ref{lem:3.2}.
\end{proof}

\begin{cor}\label{cor:3.4}
If $\mathrm{hd}(\mathfrak{a},M)\leq 1$, then $\mathrm{H}^{\mathfrak{a}}_{i}(M)$ is $\mathfrak{a}$-coartinian for all $i$ and every artinian $R$-module $M$.
\end{cor}

\begin{proof}
Since $\mathrm{hd}(\mathfrak{a},M)\leq 1$, $\mathrm{H}^{\mathfrak{a}}_{i}(M)=0$ for $i> 1$. On the other hand, $\mathrm{H}^{\mathfrak{a}}_{0}(M)=\Lambda^{\mathfrak{a}}(M)$ is artinian because $M$ is artinian. Hence $\mathrm{H}^{\mathfrak{a}}_{0}(M)$ is $\mathfrak{a}$-coartinian and therefore $\mathrm{H}^{\mathfrak{a}}_{1}(M)$ is also $\mathfrak{a}$-coartinian by Lemma \ref{lem:3.2}.
\end{proof}

\begin{lem}\label{lem:3.5}
Let $M$ be a semi-discrete linearly compact $R$-module with $\mathrm{mag}_{R}M=0$. Then $\mathrm{H}_{i}^{\mathfrak{a}}(M)$ is $\mathfrak{a}$-coartinian for $i\geq 0$.
\end{lem}

\begin{proof}
Since $M$ is semi-discrete linearly compact, we have $\mathrm{Cosupp}_{R}\mathrm{H}_{i}^{\mathfrak{a}}(M)\subseteq \mathrm{V}(\mathfrak{a})$ and $\mathrm{H}_{i}^{\mathfrak{a}}(M)=0$ for $i> 0$ by Lemma \ref{lem:3.1}. We only need to show that $\mathrm{H}_{0}^{\mathfrak{a}}(M)$ is $\mathfrak{a}$-coartinian. The proof will be divided into two steps.

{\bf Step 1.} Assume that $M$ is an artinian module. It consequently follows from Lemma \ref{lem:3.2} that $\mathrm{H}_{0}^{\mathfrak{a}}(M)$ is $\mathfrak{a}$-coartinian.

{\bf Step 2.} Suppose that $M$ is a semi-discrete linearly compact $R$-module, $\mathcal{M}$ be a base consisting of neighborhood of the zero element of $M$. By \cite[5.5]{M}, $M\cong \underleftarrow{\text{lim}}_{U\in \mathcal{M}}M/U$ and $M/U$ is an artinian $R$-module with $\mathrm{mag}_{R}M/U=0$ for all $U\in \mathcal{M}$. As $\mathrm{mag}_{R}M=0$ and $M$ is semi-discrete linearly compact, $\mathrm{Ndim}M=\mathrm{mag}_{R}M=0$, that is, $M$ is finitely generated. Hence $\underleftarrow{\text{lim}}_{U\in \mathcal{M}}M/U$ is finitely generated. As $\mathrm{H}_{0}^{\mathfrak{a}}(M/U)$ is $\mathfrak{a}$-coartinian for all $U\in \mathcal{M}$, it follows from \cite[Proposition 3.4]{CN1} that $\mathrm{H}_{0}^{\mathfrak{a}}(M)\cong \mathrm{H}_{0}^{\mathfrak{a}}( \underleftarrow{\text{lim}}_{U\in \mathcal{M}}M/U)\cong \underleftarrow{\text{lim}}_{U\in \mathcal{M}}\mathrm{H}_{0}^{\mathfrak{a}}(M/U)$ is $\mathfrak{a}$-coartinian. The proof is complete.
\end{proof}

The following result answers Question $1'$.

\begin{thm}\label{thm:3.6}
Let $(R,\mathfrak{m})$ be a local ring with the $\mathfrak{m}$-adic topology and $M$ a semi-discrete linearly compact $R$-module with $\mathrm{mag}_{R}(M)\leq 1$. Then $\mathrm{H}_{i}^{\mathfrak{a}}(M)$ is $\mathfrak{a}$-coartinian for all $i$.
\end{thm}

\begin{proof}
Since $M$ is semi-discrete linearly compact, $\mathrm{Cosupp}_{R}\mathrm{H}_{i}^{\mathfrak{a}}(M)\subseteq \mathrm{V}(\mathfrak{a})$ and $\mathrm{H}_{i}^{\mathfrak{a}}(M)=0$ for $i> 1$ by Lemma \ref{lem:3.1}. If $\mathrm{mag}_{R}M=0$, the result holds by Lemma \ref{lem:3.5}. If $\mathrm{mag}_{R}M=1$, $\mathrm{H}_{1}^{\mathfrak{a}}(M)$ is $\mathfrak{a}$-coartinian by \cite[Theorem 4.14]{N}. We only need to show that $\mathrm{H}_{0}^{\mathfrak{a}}(M)$ is $\mathfrak{a}$-coartinian. Note that $M$ is semi-discrete linearly compact, from \cite{Z} there is a short exact sequence
$$0\rightarrow N\rightarrow M\rightarrow B\rightarrow 0,$$
where $N$ is finitely generated and $B$ is artinian. Hence we get a long exact sequence of local homology modules
$$\xymatrix{\cdots\ar[r] & \mathrm{H}_{1}^{\mathfrak{a}}(B)\ar[r] & \mathrm{H}_{0}^{\mathfrak{a}}(N)\ar[r] & \mathrm{H}_{0}^{\mathfrak{a}}(M)\ar[r] &
\mathrm{H}_{0}^{\mathfrak{a}}(B)\ar[r] & 0.}$$
Since $B$ is artinian, we get $\mathrm{H}_{0}^{\mathfrak{a}}(B)$ is $\mathfrak{a}$-coartinian. On the other hand, as $N$ is finitely generated and linearly compact, we have $\mathrm{H}_{0}^{\mathfrak{a}}(N)$ is $\mathfrak{a}$-coartinian by Lemma \ref{lem:3.5}. Hence $\mathrm{H}_{0}^{\mathfrak{a}}(M)$ is $\mathfrak{a}$-coartinian.
\end{proof}

\begin{lem}\label{lem:3.7}
Let $(R,\mathfrak{m})$ be a local ring and $M$ a finitely generated $R$-module. Let $\mathfrak{a}$ be an ideal of $R$ such that the $R$-module $M/\mathfrak{a}M$ is artinian. Then $\mathrm{V}(\mathfrak{a})\cap \mathrm{Ass}_{R}M\subseteq \mathrm{V}(\mathfrak{m})$.
\end{lem}

\begin{proof}
Let $\mathfrak{p}\in \mathrm{V}(\mathfrak{a})\cap \mathrm{Ass}_{R}M$, and let
$$0=T\cap S_{1}\cap \cdots \cap S_{n}$$
be a minimal primary decomposition of $0$ in $M$, where $T$ is a $\mathfrak{p}$-primary submodule of $M$ and $S_{i}$ is a $\mathfrak{p}_{i}$-primary submodule of $M$ for $i=1,\cdots,n$.
Since $T$ is $\mathfrak{p}$-primary, it follows that there exists a positive integer $t$ such that
$\mathfrak{p}^{t}(M/T)=0$, and so $\mathfrak{a}^{t}(M/T)=0$ as $\mathfrak{a}\subseteq \mathfrak{p}$. Since $M/\mathfrak{a}M$ is artinian, $M/\mathfrak{a}^{t}M$ is artinian. On the other hand, there is a short exact sequence
$$0\rightarrow T/\mathfrak{a}^{t}M\rightarrow M/\mathfrak{a}^{t}M\rightarrow M/T\rightarrow 0.$$
Hence $M/T$ has finite length, and so $\mathfrak{p}=\mathfrak{m}$. Therefore, $\mathrm{V}(\mathfrak{a})\cap \mathrm{Ass}_{R}M\subseteq \mathrm{V}(\mathfrak{m})$ as required.
\end{proof}

\begin{lem}\label{lem:3.8}
Let $M$ be a semi-discrete linearly compact $R$-module with $\mathrm{mag}_{R}M=0$ and $\mathrm{Cosupp}_{R}M\subseteq \mathrm{V}(\mathfrak{a})$. Then $M$ is $\mathfrak{a}$-coartinian if and only if $M/\mathfrak{a}M$ is artinian.
\end{lem}

\begin{proof}
$\Rightarrow)$ It is clear.

$\Leftarrow)$ Note that $\mathrm{Cosupp}_{R}M/\mathfrak{a}M\subseteq \mathrm{Cosupp}_{R}M\cap \mathrm{V}(\mathfrak{a})=\mathrm{Cosupp}_{R}M$. Since $\mathrm{mag}_{R}M=0$, one has $\mathrm{mag}_{R}M/\mathfrak{a}M=0$. As $M/\mathfrak{a}M$ is semi-discrete linearly compact, $\mathrm{Ndim}M/\mathfrak{a}M=\mathrm{mag}_{R}M/\mathfrak{a}M=0$, that is, $M/\mathfrak{a}M$ is finitely generated. Thus $M/\mathfrak{a}M$ has finite length. It follows that $M$ is $\mathfrak{a}$-coartinian by \cite[Lemma 3.5]{N1}.
\end{proof}

Recall the arithmetic rank of an ideal $\mathfrak{b}$ in $R$, denoted by $\mathrm{ara}(\mathfrak{b})$, is the least number of elements of $R$ required to generate an ideal which has the same radical as $\mathfrak{b}$, i.e.,
$$\mathrm{ara}(\mathfrak{b})=\min \{n\in \mathbb{N_{\circ}}: \exists b_{1},\cdots b_{n}\in R\ \mathrm{with}\ \mathrm{Rad}(b_{1},\cdots b_{n})=\mathrm{Rad}(\mathfrak{b})\}.$$
\noindent{}Let $M$ be an $R$-module. The arithmetic rank of an ideal $\mathfrak{b}$ in $R$ with respect to $M$, denoted by $\mathrm{ara}_{M}(\mathfrak{b})$, is defined by the arithmetic rank of an ideal $\mathfrak{b}+\mathrm{Ann}_{R}M/\mathrm{Ann}_{R}M$ in the ring $R/\mathrm{Ann}_{R}M$.

The next is the second main result, which is useful in answering Question $2'$.

\begin{thm}\label{thm:3.9}
Let $M$ be a semi-discrete linearly compact $R$-module with $\mathrm{mag}_{R}M\leq 1$ and $\mathrm{Cosupp}_{R}M\subseteq \mathrm{V}(\mathfrak{a})$. Then $M$ is $\mathfrak{a}$-coartinian if and only if $M/\mathfrak{a}M$ and $\mathrm{Tor}_{1}^{R}(R/\mathfrak{a},M)$ are artinian.
\end{thm}

\begin{proof}
$\Rightarrow)$ It is clear.

$\Leftarrow)$ When $\mathrm{mag}_{R}M=0$, it is by Lemma \ref{lem:3.8}. We may assume $\mathrm{mag}_{R}M=1$ and use induction on
$$t=\mathrm{ara}_{M}(\mathfrak{a})=\mathrm{ara}(\mathfrak{a}+\mathrm{Ann}_{R}M/\mathrm{Ann}_{R}M)$$
to prove that $M$ is $\mathfrak{a}$-coartinian. If $t=0$, then it follows from definition that $\mathfrak{a}^{n}\subseteq \mathrm{Ann}_{R}M$ for some positive integer $n$, and $M/\mathfrak{a}^{n}M \cong M$. Since $M/\mathfrak{a}M$ is artinian, we get $M/\mathfrak{a}^{n}M \cong M$ is artinian. Hence $M$ is $\mathfrak{a}$-coartinian. So assume that $t>0$, and the result has been proved for $i\leq t-1$. Let
$$\mathcal{T}:=\{\mathfrak{p}\in \mathrm{Cosupp}_{R}M\mid \mathrm{dim}R/\mathfrak{p}=1\}$$
As $\mathrm{Coass}_{R}M/\mathfrak{a}M=\mathrm{V}(\mathfrak{a})\cap \mathrm{Coass}_{R}M=\mathrm{Coass}_{R}M$ and $M/\mathfrak{a}M$ is artinian, it follows that the set $\mathrm{Coass}_{R}M$ is finite. Hence $\mathcal{T}$ is finite. Moreover, as for each $\mathfrak{p}\in \mathcal{T}$, $\mathrm{Cosupp}_{R_{\mathfrak{p}}}\mathrm{Hom}_{R}(R_{\mathfrak{p}},M)\\ \subseteq \mathrm{V}(\mathfrak{p}R_{\mathfrak{p}})$, one has $\mathrm{mag}_{R_{\mathfrak{p}}}\mathrm{Hom}_{R}(R_{\mathfrak{p}},M)=0$. Since $\mathrm{Hom}_{R}(R_{\mathfrak{p}},M)$ is a semi-discrete linearly compact $R$-module, there is a short exact
$$0\rightarrow K\rightarrow \mathrm{Hom}_{R}(R_{\mathfrak{p}},M)\rightarrow A\rightarrow 0,$$
where $K$ is a finitely generated $R$-module and $A$ is an artinian $R$-module. It induces an exact sequence
$$0\rightarrow \mathrm{Hom}_{R}(R_{\mathfrak{p}},K)\rightarrow \mathrm{Hom}_{R}(R_{\mathfrak{p}},\mathrm{Hom}_{R}(R_{\mathfrak{p}},M))\rightarrow \mathrm{Hom}_{R}(R_{\mathfrak{p}},A)\rightarrow0.$$
Since $K$ is finitely generated, we have $\mathrm{Hom}_{R}(R_{\mathfrak{p}},K)=0$. Thus $\mathrm{Hom}_{R}(R_{\mathfrak{p}},M)\cong \mathrm{Hom}_{R}(R_{\mathfrak{p}},A)$. It follows from \cite[Theorem 3.2]{M0} that $\mathrm{Hom}_{R}(R_{\mathfrak{p}},A)$ is annihilated by $\mathfrak{p}R_{\mathfrak{p}}$, then $\mathrm{Hom}_{R}(R_{\mathfrak{p}},A)$ is artinian by \cite[Theorem 7.30]{S1}. Hence $\mathrm{Hom}_{R}(R_{\mathfrak{p}},M)/\mathfrak{a}R_{\mathfrak{p}}\mathrm{Hom}_{R}(R_{\mathfrak{p}},M)$ is artinian. Let
$$\mathcal{T}=\{\mathfrak{p}_{1},\cdots,\mathfrak{p}_{s}\}.$$
By Lemma \ref{lem:3.7}, we have
$$\mathrm{V}(\mathfrak{a}R_{\mathfrak{p}_{j}})\cap \mathrm{Ass}_{R_{\mathfrak{p}_{j}}}\mathrm{Hom}_{R}(R_{\mathfrak{p}_{j}},M)\subseteq \mathrm{V(\mathfrak{p}_{j}R_{\mathfrak{p}_{j}})},$$
for all $j=1,2,\cdots,s$. Next, let
$$\mathcal{U}:=\mathop{\bigcup^{n}}_{j=1}\{\mathfrak{q}\in \mathrm{Spec}R\mid \mathfrak{q}R_{\mathfrak{p}_{j}}\in \mathrm{Ass}_{R_{\mathfrak{p}_{j}}}\mathrm{Hom}_{R}(R_{\mathfrak{p}_{j}},M)\}.$$
Then it is easy to see that $\mathcal{U}\cap \mathrm{V}(\mathfrak{a})\subseteq \mathcal{T}$. On the other hand, since $t=\mathrm{ara}_{M}(\mathfrak{a})\geq 1$, there exists elements $y_{1},\cdots,y_{t}\in \mathfrak{a}$ such that
$$\mathrm{Rad}(\mathfrak{a}+\mathrm{Ann}_{R}M/\mathrm{Ann}_{R}M)=\mathrm{Rad}((y_{1},\cdots,y_{t})+\mathrm{Ann}_{R}M/\mathrm{Ann}_{R}M).$$
Now, as $\mathfrak{a}\nsubseteq \bigcup_{\mathfrak{q}\in \mathcal{U}\backslash \mathrm{V}(\mathfrak{a})}\mathfrak{q}$, it follows that
$(y_{1},\cdots,y_{t})+\mathrm{Ann}_{R}M\nsubseteq \bigcup _{\mathfrak{q}\in \mathcal{U}\backslash \mathrm{V}(\mathfrak{a})}\mathfrak{q}$. On the other hand, for each $\mathfrak{q}\in \mathcal{U}$, we have
$$\mathfrak{q}R_{\mathfrak{p}_{j}}\in \mathrm{Ass}_{R_{\mathfrak{p}_{j}}}\mathrm{Hom}_{R}(R_{\mathfrak{p}_{j}},M),$$
for some integer $1\leq j\leq n$. Hence
$$\mathrm{Ann}_{R}(M)R_{\mathfrak{p}_{j}}\subseteq \mathrm{Ann}_{R_{\mathfrak{p}_{j}}}\mathrm{Hom}_{R}(R_{\mathfrak{p}_{j}},M)\subseteq \mathfrak{q}R_{\mathfrak{p}_{j}}$$
and so $\mathrm{Ann}_{R}M\subseteq \mathfrak{q}$. Consequently, it follows from
$$\mathrm{Ann}_{R}M\subseteq \mathop{\bigcap}_{\mathfrak{q}\in \mathcal{U}\backslash \mathrm{V}(\mathfrak{a})}\mathfrak{q}$$
that $(y_{1},\cdots,y_{t})\nsubseteq \bigcup_{\mathfrak{q}\in \mathcal{U}\backslash \mathrm{V}(\mathfrak{a})}q$. By \cite[Theorem 16.8]{M5} there is $a\in (y_{2},\cdots,y_{t})$ such that $y_{1}+a\notin \bigcup_{\mathfrak{q}\in \mathcal{U}\backslash \mathrm{V}(\mathfrak{a})}\mathfrak{q}$. Let $x:=y_{1}+a$. Then $x\in \mathfrak{a}$ and
$$\mathrm{Rad}(\mathfrak{a}+\mathrm{Ann}_{R}M/\mathrm{Ann}_{R}M)=\mathrm{Rad}((x,y_{2},\cdots,y_{t})+\mathrm{Ann}_{R}M/\mathrm{Ann}_{R}M).$$
Next, let $N:=M/xM$. Then, it is easy to see that
$$\mathrm{ara}_{N}(\mathfrak{a})=\mathrm{ara}(\mathfrak{a}+\mathrm{Ann}_{R}N/\mathrm{Ann}_{R}N)\leq t-1$$
\noindent{}(note that $x\in \mathrm{Ann}_{R}N$), and
$$\mathrm{Rad}(\mathfrak{a}+\mathrm{Ann}_{R}N/\mathrm{Ann}_{R}N)=\mathrm{Rad}((y_{2},\cdots,y_{t})+\mathrm{Ann}_{R}N/\mathrm{Ann}_{R}N).$$
\noindent{}Now, the exact sequence
$$\xymatrix@C=0.5cm{
  0 \ar[r] & xM \ar[r]^{} & M \ar[r]^{} & N \ar[r] & 0 }$$
\noindent{}induces an exact sequence
$$\xymatrix@C=0.5cm{
 & \mathrm{Tor}^{R}_{1}(R/\mathfrak{a},M)\ar[r] & \mathrm{Tor}^{R}_{1}(R/\mathfrak{a},N)\ar[r] & xM/\mathfrak{a}(xM)\ar[r] & M/\mathfrak{a}M\ar[r] & N/\mathfrak{a}N \ar[r] & 0 },$$
which implies that the $R$-modules $N/\mathfrak{a}N$ and $\mathrm{Tor}^{R}_{1}(R/\mathfrak{a},N)$ are artinian.
Consequently, by the inductive hypothesis, the $R$-module $N$ is $\mathfrak{a}$-coartinian. Moreover, the exact sequence
$$\xymatrix@C=0.5cm{
  0\ar[r] & xM\ar[r] & M\ar[r] & N\ar[r] & 0 }$$
induces an exact sequence
$$\xymatrix@C=0.5cm{
 \mathrm{Tor}^{R}_{2}(R/\mathfrak{a},N)\ar[r] & \mathrm{Tor}^{R}_{1}(R/\mathfrak{a},xM)\ar[r] & \mathrm{Tor}^{R}_{1}(R/\mathfrak{a},M)},$$
which implies that the $R$-module $\mathrm{Tor}^{R}_{1}(R/\mathfrak{a},xM)$ is artinian. Also, from the exact sequence
$$\xymatrix@C=0.5cm{
  0\ar[r] & (0:_{M}x)\ar[r] & M\ar[r] & xM\ar[r] & 0 },$$
let $L:=(0:_{M}x)$, we get the exact sequence
$$\xymatrix@C=0.5cm{
 \mathrm{Tor}^{R}_{1}(R/\mathfrak{a},xM)\ar[r] & L/\mathfrak{a}L\ar[r] & M/\mathfrak{a}M},$$
which implies that the $R$-module $L/\mathfrak{a}L$ is artinian. On the other hand, since $M$ is semi-discrete linearly compact, $L$ is also semi-discrete linearly compact. It follows from \cite[Lemma 3.5]{N1} that $L$ is $\mathfrak{a}$-coartinian. Since the $R$-modules $L=(0:_{M}x)$ and $M/xM$ are $\mathfrak{a}$-coartinian, it follows from Proposition \ref{prop:2.5} that $M$ is $\mathfrak{a}$-coartinian. This completes the inductive step.
\end{proof}

\begin{cor}\label{cor:3.10}
Let $\mathfrak{a}$ be an ideal of $R$, $\mathcal{C}(R,\mathfrak{a})^{1}_{coa}$ denote the category of semi-discrete linearly compact $\mathfrak{a}$-coartinian $R$-modules $M$ with $\mathrm{mag}_{R}M\leq 1$. Then $\mathcal{C}(R,\mathfrak{a})^{1}_{coa}$ is an Abelian category.
\end{cor}

\begin{proof}
Let $M$, $N\in \mathcal{C}(R,\mathfrak{a})_{coa}^{1}$ and $f$: $M\rightarrow N$ be an $R$-homomorphism. It is enough to show that the $R$-modules $\mathrm{ker}f$ and $\mathrm{coker}f$ are $\mathfrak{a}$-coartinian. To this end, the exact sequence
$$0\rightarrow \mathrm{Im}f\rightarrow N\rightarrow \mathrm{coker}f\rightarrow 0,$$
induces a long exact sequence
$$\mathrm{Tor}_{1}^{R}(R/\mathfrak{a},N)\rightarrow \mathrm{Tor}_{1}^{R}(R/\mathfrak{a},\mathrm{coker}f)\rightarrow \mathrm{Im}f/\mathfrak{a}\mathrm{Im}f\rightarrow N/\mathfrak{a}N\rightarrow \mathrm{coker}f/\mathfrak{a}\mathrm{coker}f\rightarrow 0.$$
It implies that $\mathrm{coker}f/\mathfrak{a}\mathrm{coker}f$ and $\mathrm{Tor}_{1}^{R}(R/\mathfrak{a},\mathrm{coker}f)$ are artinian. Therefore $\mathrm{coker}f$ is $\mathfrak{a}$-coartinian by Theorem \ref{thm:3.9}. Now, the assertion follows from the following exact sequences
$$0\rightarrow \mathrm{ker}f\rightarrow M\rightarrow \mathrm{Im}f\rightarrow 0,$$
and
$$0\rightarrow \mathrm{Im}f\rightarrow N\rightarrow \mathrm{coker}f\rightarrow 0.$$
\end{proof}

Now, we begin to prove the last main results of this section which answers the second question on $\mathfrak{a}$-coartinian $R$-modules. We denote $\mathcal{C}(R,\mathfrak{a})_{coa}$ the category of semi-discrete linearly compact $\mathfrak{a}$-coartinian $R$-modules.

\begin{cor}\label{cor:3.11}
Let $\mathfrak{a}$ be an ideal of $R$ with $\mathrm{dim}R/\mathfrak{a}\leq 1$. Then $\mathcal{C}(R,\mathfrak{a})_{coa}$ is an Abelian subcategory of the category of all $R$-modules.
\end{cor}

\begin{proof}
Let $M$ be an arbitrary $R$-module in $\mathcal{C}(R,\mathfrak{a})_{coa}$. Since $\mathrm{dim}R/\mathfrak{a}\leq 1$ and $\mathrm{Cosupp}_{R}M\subseteq \mathrm{V}(\mathfrak{a})$, we have $\mathrm{mag}_{R}M\leq 1$. Hence the result holds by Corollary \ref{cor:3.10}.
\end{proof}

\begin{thm}\label{thm:3.12}
Let $\mathfrak{a}$ be an ideal of $R$ such that $\mathrm{cd}(\mathfrak{a},R)\leq 1$. Then $\mathcal{C}(R,\mathfrak{a})_{coa}$ is an Abelian subcategory of the category of all $R$-modules.
\end{thm}

\begin{proof}
Let $M$, $N\in \mathcal{C}(R,\mathfrak{a})_{coa}$ and $f$: $M\rightarrow N$ be an $R$-homomorphism. It is enough to show that the $R$-modules $\mathrm{ker}f$ and $\mathrm{coker}f$ are $\mathfrak{a}$-coartinian. To this end, the exact sequence
$$0\rightarrow \mathrm{Im}f\rightarrow N\rightarrow \mathrm{coker}f\rightarrow 0,$$
induces an exact sequence
$$0\rightarrow \mathrm{Hom}_{R}(R/\mathfrak{a},\mathrm{Im}f)\rightarrow \mathrm{Hom}_{R}(R/\mathfrak{a},N).$$
Since $N$ is $\mathfrak{a}$-coartinian, then $\mathrm{Hom}_{R}(R/\mathfrak{a},N)$ is artinian by Theorem \ref{thm:2.3}, it implies that $\mathrm{Hom}_{R}(R/\mathfrak{a},\mathrm{Im}f)$ is also artinian. Now, the exact sequence
$$0\rightarrow \mathrm{ker}f\rightarrow M\rightarrow \mathrm{Im}f\rightarrow 0,$$
induces an exact sequence
$$0\rightarrow \mathrm{Hom}_{R}(R/\mathfrak{a},\mathrm{ker}f)\rightarrow \mathrm{Hom}_{R}(R/\mathfrak{a},M)\rightarrow \mathrm{Hom}_{R}(R/\mathfrak{a},\mathrm{Im}f)$$
$$\rightarrow \mathrm{Ext}_{R}^{1}(R/\mathfrak{a},\mathrm{ker}f)\rightarrow \mathrm{Ext}_{R}^{1}(R/\mathfrak{a},M).$$
From the hypothesis, $\mathrm{Hom}_{R}(R/\mathfrak{a},M)$ and $\mathrm{Ext}_{R}^{1}(R/\mathfrak{a},M)$ are artinian, which implies that the $R$-modules $\mathrm{Hom}_{R}(R/\mathfrak{a},\mathrm{ker}f)$ and $\mathrm{Ext}_{R}^{1}(R/\mathfrak{a},\mathrm{ker}f)$ are artinian. Therefore, it follows from \cite[Lemma 2.1]{PAB} and Theorem \ref{thm:3.9} that the $R$-module $\mathrm{ker}f$ is $\mathfrak{a}$-coartinian. Now, the following exact sequences
$$0\rightarrow \mathrm{ker}f\rightarrow M\rightarrow \mathrm{Im}f\rightarrow 0,$$

$$0\rightarrow \mathrm{Im}f\rightarrow N\rightarrow \mathrm{coker}f\rightarrow 0.$$
yields the desired assertion.
\end{proof}

\begin{cor}\label{cor:3.13}
Let $\mathfrak{a}$ be an ideal of $R$ such that $\mathrm{dim}R/\mathfrak{a}\leq1$ or $\mathrm{cd}(\mathfrak{a},R)\leq 1$ and $M$ an $\mathfrak{a}$-coartinian $R$-module. Then the $R$-module $\mathrm{Ext}_{R}^{i}(N,M)$ is $\mathfrak{a}$-coartinian, for all finitely generated $R$-modules $N$ and all integers $i\geq0$.
\end{cor}

\begin{proof}
Since $N$ is finitely generated, we get $N$ has a free resolution $\cdots\rightarrow F_{i}\rightarrow F_{i-1}\rightarrow \cdots\rightarrow F_{1}\rightarrow F_{0}\rightarrow 0$ with each $F_{i}$ is finitely generated free $R$-module. Now the assertion follows using Corollary \ref{cor:3.11} or Theorem \ref{thm:3.12} and computing the $R$-module $\mathrm{Ext}_{R}^{i}(N,M)$, using this free resolution.
\end{proof}

\bigskip \centerline {\bf ACKNOWLEDGEMENTS} This research was partially supported by National Natural Science Foundation of China (11761060, 11901463), Innovation Ability Enhancement Project of Gansu Higher Education Institutions (2021A-002) and Outstanding Graduate Student ``Star of Innovation" Project of Education Department of Gansu Province (2021CXZX-179).

\end{document}